\documentclass{amsart}
\usepackage[utf8]{inputenc}
\usepackage{latexsym}
\usepackage{amsfonts}
\usepackage[all]{xy}
\usepackage{color}
\usepackage{amssymb, mathrsfs, amsfonts, amsmath}
\usepackage{amsbsy}
\newtheorem{theorem}{Theorem}[section]
\newtheorem{thm}{Theorem}[section]
 \newtheorem{cor}[thm]{Corollary} 
 \newtheorem{lem}[thm]{Lemma}
 \newtheorem{remark}[thm]{Remark}

 \newcommand{\grad}{\text{grad}\,}
 \newcommand{\diver}{{{\rm{div}}}}
 \newcommand\raisepunct[1]{\,\mathpunct{\raisebox{0.5ex}{#1}}}

 \newtheorem{ex}{Example}
 
 \numberwithin{equation}{section}
 
 \newcommand{\hess}{\textrm{Hess}\,}
 \newcommand{\Sec}{\textrm{Sec}}

 \newcommand{\dist}{{\rm dist}}

\begin{document}


\title[Mean exit times from submanifolds]{Mean exit times from  submanifolds \\with bounded mean curvature} 

\author[G. P. Bessa]{G. Pacelli Bessa}
\address{Department of Mathematics, Universidade Federal do Cear\'{a} - UFC, Fortaleza, Brazil.}
\email{bessa@mat.ufc.br}

\author[S. Markvorsen]{Steen Markvorsen} 
\address{DTU Compute, Technical University of Denmark, Lyngby, Denmark.}
\curraddr{}
\email{stema@dtu.dk}
\author[L.F. Pessoa]{Leandro F. Pessoa} 
\address{Department of Mathematics,
Universidade Federal do Piauí - UFPI, Teresina - Piauí, Brazil.}
\email{leandropessoa@ufpi.edu.br}

 \thanks{Research partially supported by Alexander von Humboldt Foundation and Capes-Brazil (Finance Code 001), by Humboldt Return Fellowship BRA1213431HFSTCAPES-E, and by CNPq-Brazil, Grants 303057/2018-1, 306738/2019-8, and 422900/2021-4.}


\maketitle

\begin{abstract}
We show that submanifolds with infinite mean exit time can not be isometrically and minimally immersed into cylinders, horocylinders, cones, and wedges of some product spaces. Our approach is not based on the weak maximum principle at infinity, and thus it permits us to generalize previous results concerning non-immersibility of stochastically complete submanifolds. We also produce estimates for the complete tower of moments for submanifolds with small mean curvature immersed into cylinders.

\vspace{.2cm}


\end{abstract}

\section{Introduction}
An interesting problem in minimal surface theory  stemming from the attempts to settle the Calabi conjecture on the existence of bounded minimal surfaces \cite{calabi} is the following: 
\textit{ \hspace{-2mm}what is the largest open subset $\Omega \subset \mathbb{R}^n_{\raisepunct{,}}$ $n \geq 3$, into which a  complete Riemannian $m$-manifold $M$ satisfying some extra geometric properties can not be isometrically and minimally immersed?} 

This question has been studied by several authors in the literature, 
for instance, as an application of his maximum principle at infinity, H. Omori \cite{omori} proved that complete Riemannian manifolds with sectional curvature bounded from below can not be isometrically and minimally immersed into any non-degenerate Euclidean cone. By non-degenerate Euclidean cones we mean those cones which strictly lie in a half-space of $\mathbb{R}^n_{\raisepunct{.}}$  L. Jorge and F. Xavier   \cite{jorge_xavier_mathz},  proved Omori's non-immersibility  result  in  Euclidean   cones in the class of manifolds with bounded scalar curvature. Moreover, their techniques were robust enough to 
prove the non-immersibility of minimal  surfaces with bounded Gaussian curvature into quite general regions of $\mathbb{R}^3$ such as
 $$A_f = \{(x,y,z) \in \mathbb{R}^3 \colon z > f(y)\},$$
where $f \colon \mathbb{R} \rightarrow \mathbb{R}_+$ is a positive proper function. In the same paper \cite[Thm.1]{jorge_xavier_mathz} the authors considered complete manifolds $M$ whose scalar curvature is bounded from below and showed that they can not be isometrically immersed with small mean curvature\footnote{The mean curvature is assumed smaller than a bound depending on the radius of the ball and the sectional curvature of $\overline{M}$.}  into a normal ball of a general Riemannian manifold $\overline{M}$.

After the extension of Omori's maximum principle at infinity by S.T. Yau  \cite{yau1}, and the subsequent generalizations in \cite{cheng-yau,  pigola-rigoli-setti,ratto-rigoli-setti}, S. Pigola, A.G. Setti and M. Rigoli \cite{pigola-rigoli-setti} proved that  complete minimal Riemannian submanifolds satisfying the Omori-Yau maximum principle can not be isometrically immersed into  non-degenerate cones of $\mathbb{R}^n_{\raisepunct{.}}$ The set of geometric conditions to guarantee the validity of the Omori-Yau maximum principle on minimal submanifolds $\varphi \colon M \rightarrow\mathbb{R}^n$ includes properness of the map $\varphi$ and  Ricci curvature with at most quadratic decay  $\text{Ric}_M(x)\geq -c^2\rho^2(x)$, where  $\rho(x)={\rm dist}(x_0,x)$ is the distance function in $M$, among others. Extensions of these results are given in \cite[Thms.1 \& 2]{jorge_xavier_mathz}, see also \cite{jorge-koutrofiotis, kasue}.

In  \cite{prs_pams},
 Pigola, Rigoli and Setti established a surprising equivalence between the validity of a weak form of the Omori-Yau maximum principle, called the weak maximum principle at infinity, and the stochastic completeness of the manifold. This equivalence was shown to hold for more general elliptic operators and gave new strength to the research on this topic, see \cite{alias_mastrolia_rigoli}.
 
 For instance,
 L. Alias, P. Mastrolia and M. Rigoli  \cite[Thm.5.1]{alias_mastrolia_rigoli} proved the non-immersibility of complete minimal submanifolds $M$ into non-degenerate cones of $\mathbb{R}^{n}$ only assuming the stochastic completeness of $M$. 
 However, in the limit case, it is possible to construct stochastically complete minimal surfaces in  half-spaces of $\mathbb{R}^{3}$. As observed by A. Atsuji in \cite{atsuji}, minimal surfaces   lying between  parallel planes in $\mathbb{R}^{3}$ -- as in Jorge-Xavier's construction -- can in fact be constructed so that they are stochastically complete. In \cite{alias_bessa_dajczer}, L. Alias, G.P. Bessa and M. Dajczer considered geodesically and stochastically complete submanifolds with small mean curvature and showed that they can not be isometrically immersed into a cylindrical region of a Riemannian product $N\times \mathbb{R}^\ell$ (see also \cite{bessa_lima_pessoa, kasue}). This result was also considered for horocylindrically bounded submanifolds of $N\times L$, where $N$ is a Cartan-Hadamard manifold with negatively pinched sectional curvature, and $L$ has sectional curvature bounded from below (c.f. \cite{bessa_lira_pigola_setti}).
 
 Recalling that every stochastically complete manifold has infinite mean exit time, G.P. Bessa and J.F. Montenegro \cite{bessa_montenegro_isop} improved a general bound obtained by S. Markvorsen \cite{markvorsen} for the mean exit time of minimal submanifolds of $\mathbb{R}^n_{\raisepunct{,}}$ and proved that a complete minimal submanifold $M$ can not be isometrically immersed in a cylinder of $N\times \mathbb{R}$ if $M$ has infinite mean exit time. Recall that  the first exit time of Brownian motion $X_t$ of a domain $D \subseteq M$, $X_0=x \in D$,  is defined as
\begin{equation} \tau_{\!_D} = \inf \{t\geq 0 \colon X_t \notin D \}.\label{firstexit}\end{equation}
The mean exit time  of $D$ is then defined as the expectation $E_{D}(x)=\mathbb{E}^x[\tau_{\!_D}]$  with respect to the probability measure $\mathbb{P}^x$ weighting those Brownian paths starting at $x$. Although there is an equivalence between stochastic completeness and infinite mean exit time on model manifolds, there are stochastically incomplete manifolds with infinite mean exit time due to the non-isotropy of this latter property  (cf. \cite{bessa-pigola-setti,pessoa-pigola-setti}).

The main purpose of this work is to address this non-immersibility problem for submanifolds with infinite mean exit time through a unified elliptic point of view. Our strategy allows us to generalize recent non-immersability results from \cite{alias_bessa_dajczer,bessa_montenegro_isop} for submanifolds cylindrically bounded, from \cite{bessa_lira_pigola_setti} for submanifolds horocylindrically bounded, and from \cite{alias_mastrolia_rigoli} for submanifolds minimally immersed into cones, as well as immersed in higher dimensional wedges. Moreover, we also prove estimates for the complete tower of moments from submanifolds cylindrically bounded.
 

\section{Main results}

Let $\varphi \colon M^{m} \to N^{n-\ell}\times L^{\ell}$ be an isometric immersion of a complete $m$-dimensional Riemannian manifold into the Riemannian product $N^{n-\ell}\times L^{\ell}_{\raisepunct{,}}$ where $m\geq \ell+1$, $n-\ell \geq 2$, and suppose that the radial  sectional curvature of $N$ is bounded above $\Sec^{\rm rad}_{N}(x)\leq - G(\rho_{_N}(x))\leq 0 $, where $\rho_{_N}(x)={\rm dist}_N(x,o)$ and $G\in C^{0}(\mathbb{R}_{+})$. We 
 let $h \colon [0,+\infty) \rightarrow \mathbb{R}$ be a $C^2$-solution of the following initial value problem
\begin{equation}
\left\lbrace\begin{array}{l}\label{cauchy_problem_h_intro}
    h'' - Gh = 0,  \\[0.2cm]
    h(0) = 0,\,\, h'(0) = 1,
\end{array}\right.
\end{equation} and let
$\pi_{1} \colon N^{n-\ell}\times L^{\ell}\to N^{n-\ell}$ be  the projection onto the $N$ factor. Likewise, $\pi_{2}$ will denote the projection onto $L$.

In our first result, we consider cylindrically bounded submanifolds with small mean curvature and generalize  results from \cite{alias_bessa_dajczer,bessa_montenegro_isop}. 
\begin{theorem}\label{thmMTE_intro}
Let $\varphi \colon M^{m} \rightarrow N^{n-\ell}\times L^{\ell}$ be a complete Riemannian $m$-submanifold with  $\Sec^{\rm rad}_{N}(x)\leq - G(\rho_{_N}(x))\leq 0 $.  Let $D \subset M$ be a relatively compact open set such that $\pi_1(\varphi(D)) \subset  B_{N}(r_{\!_D})$, where $B_N(r_{\!_D})$ is a normal geodesic ball. If
\begin{equation}\label{hyp_curv_media_intro}
\max_{D}\vert H\vert \leq (m-\ell-\eta)\frac{h'}{h}(r_{\!_D}),
\end{equation} 
for some $h$ satisfying \eqref{cauchy_problem_h_intro} and $\eta > 0$, then   the mean exit time from $D$ is bounded above as follows
\begin{equation}E_{D}\leq \int_{0}^{r_{\!_D}}\frac{1}{h^{\eta - 1}(\tau)}\int_{0}^{\tau}h^{\eta - 1}(\xi)d\xi d\tau.\label{eqMET1}
\end{equation} 
In particular, if $N$ is a Cartan-Hadamard manifold and $M$ is cylindrically bounded, then $M$ has finite mean exit time $E_{M}<\infty$.
\end{theorem}

\begin{remark}\label{rmk1b} When  $\varphi \colon M^m \rightarrow N^{n-\ell}\times L^{\ell}$  is minimal, we may take $\eta=m-\ell$, and the right hand side of \eqref{eqMET1} is the mean exit time of the ball of radius $r_{\!_D}$ in the $(m-\ell)$-dimensional model space $(\mathbb{R}^{n-\ell} \, , ds^2=dt^2+h^2(t)d\theta^2)$ of curvature  $\Sec^{\rm rad}_{_{\mathbb{M}_{h}^{m}}}=-h''/h$.
\end{remark}

As a scholium of Theorem \ref{thmMTE_intro} we extend the result in \cite{bessa_lira_pigola_setti} for horocylindrically bounded submanifolds of $N\times L$ without assuming any curvature bounds on the fiber $L$ and without any lower bound for the sectional curvature of $N$. We  denote by $\mathfrak{B}_{\sigma,r} = \{ x \in N \colon b_{\sigma}(x) \leq r\}$ the horoball defined by a Busemann function $b_\sigma$. 

\begin{theorem}\label{thm_scholium_intro}
Let $\varphi \colon M^m \rightarrow N^{n-\ell}\times L^{\ell}$ be a complete Riemannian submanifold and assume that $N$ is Cartan-Hadamard with $\Sec_N \leq -b< 0$. For each compact set $D \subset M$ such that $\pi_{1}(\varphi(D)) \subset  \mathfrak{B}_{\sigma,r_{\!_D}}$, we assume that
\begin{eqnarray}\label{hyp_curv_media_busem_intro}
\max_{D}\vert H\vert < (m-\ell)\sqrt{b}\,\coth(\sqrt{b}\,r_{\!_D}).
\end{eqnarray}
Then there exists a constant $C=C(m,b,r_{\!_D}, \sup_{M}\vert H\vert)$, such that $E_{D}\leq C$. In particular, if $M$ is horocylindrically bounded, then $M$ has finite mean exit time $E_{M}<\infty$.
\end{theorem}

On a complete manifold $M_{\!\raisepunct{,}}$ finiteness of the mean exit time corresponds to the integrability of the positive Green kernel. On the other hand,  in \cite{grigoryan_summable}, A. Grigor'yan proved that the non-integrability of the Green kernel is equivalent to the validity of the $L^1$-Liouville property for superharmonic functions on $M$: every non-negative, superharmonic function $u \in L^1(M)$ must be constant. Motivated by these relations, the authors in \cite[Prop.43]{pessoa-pigola-setti} extended this cycle of ideas to consider a Dirichlet $L^1$-Liouville property on manifolds with boundary, and investigated the mean exit time of cones under Dirichlet boundary condition.\footnote{ We refer the interested readers to Appendix II for a slight generalization of \cite[Prop.43]{pessoa-pigola-setti} and a dual version of \cite[Thm.44]{pessoa-pigola-setti}.} Given  a closed manifold $L^{\ell-1}$, for $\ell \geq 2$, an $\ell$-dimensional cone over a domain $\Omega \subset L$ is the subset  
$$C_\Omega = \{(r,\theta) \in [0, \infty) \times_{r} L^{\ell-1} \colon \theta \in \Omega \}. $$
In what follows we denote by $\lambda_1(\Omega)$ the smallest eigenvalue of the Laplace-Beltrami operator $-\triangle$ with respect to the domain $\Omega \subset L$.

\begin{theorem}[Pessoa-Pigola-Setti]\label{prop_eigenvalue_intro}
Let $C_\Omega$ be a cone over a domain $\Omega \subset L$ in the warped product $[0,+\infty) \times_{r} L^{\ell-1}$. Then $C_\Omega$ has finite Dirichlet mean exit time if and only if 
\begin{equation}\label{eigen_main_hyp_intro}
\lambda_1(\Omega) > 2\ell.
\end{equation}
\end{theorem}

Now let $\varphi \colon M^m\! \rightarrow \mathbb{R}^{n}=[0, \infty)\times_r \mathbb{S}^{n-1}$ be a minimal immersion of an  $m$-manifold $M$, and assume that $\varphi (M)\subset [0,\infty)  \times_r \Omega_\theta^{n-1}$ for some geodesic ball $\Omega_{\theta}^{n-1}\subset \mathbb{S}^{n-1}$ of radius $\theta$, centred at the north pole. By Theorem \ref{prop_eigenvalue_intro}, the mean exit time of the cone $C_{\Omega_{\theta}^{n-1}}=[0,\infty)  \times_r \Omega_\theta^{n-1}$ is finite if and only if $\lambda_1(\Omega_{\theta}^{n-1})>2n$. 
However, in view of the comparison theorems from \cite{bessa_montenegro_isop,markvorsen} the expected dimension of comparison should be $m-1$. Therefore one might expect  that the finiteness of the mean exit time of $M$ depends on  the first Dirichlet eigenvalue of a geodesic ball in $\mathbb{S}^{m-1}$. 

In our second contribution we are going to provide a sufficient condition for the finiteness of the mean exit time on minimal submanifolds contained in cones of $\mathbb{R}^n$ in terms of an upper bound for the width of the cone $[0, \infty)\times_r \Omega_{\theta\raisepunct{,}}^{n-1}$ instead of the first Dirichlet eigenvalue $\lambda_1(\Omega_\theta^{n-1})$. The techniques developed to estimate the mean exit time of minimal submanifolds bounded by a non-degenerate cone are robust enough to consider immersed  minimal submanifolds into non-degenerate wedges of $\mathbb{R}^{n+2}$ for $n\geq 2$.

Let $M^m$ be a minimal submanifolds isometrically immersed into the Riemannian product space $N^n \times L^\ell \times P^k$, where $N,L$ and $P$ are complete Riemannian manifolds and $n\geq 2$, $\ell\geq 1$ and $k\geq 0$. Given $\alpha> 0$ let us consider the wedge 
$$W_\alpha \doteq \{ (x,y,z) \in N \times L \times P \colon \rho_N(x) \leq \alpha \rho_L(y) \},$$
where $\rho_N(x)$ and $\rho_L(y)$ denote the distance functions to the origins $o_N \in N$ and $o_L \in L$, respectively.

\begin{theorem}\label{thm_wedge_intro}
Let $\varphi \colon M^m \rightarrow N^n \times L^\ell \times P^k$ be a minimal submanifold with $n\geq 2$, $\ell\geq 1$ and $k\geq 0$. Assume that  $\Sec_N \leq 0 \leq \Sec_L$, and $\varphi(M) \subset W_\alpha$, for some $\alpha>0$. If $m - (\alpha^2+1)\ell - k > 0$, then $M$ has finite mean exit time $E_{M}<\infty$.
\end{theorem}

\begin{remark}{\rm The condition $m - (\alpha^2+1)\ell - k > 0$ does not hold when $m=2$ and  $n=\ell=k=1$. Therefore, Theorem \ref{thm_wedge_intro} does not apply to complete minimal surfaces immersed in non-degenerated wedges of $\mathbb{R}^{3}_{\raisepunct{.}}$  However, minimal surfaces immersed in any non-degenerate wedge of $\mathbb{R}^3$ are stochastically incomplete (cf. \cite{boscacci}).}
\end{remark}

As a direct consequence of Theorem \eqref{thm_wedge_intro} (case $k=0$), we can consider minimal immersions into non-degenerate Euclidean cones.

\begin{cor}\label{cor_cone_intro}
Let $\varphi \colon M^m \rightarrow \mathbb{R}^{n+1}$, $n\geq 2$, be a minimal submanifold immersed in a non-degenerated cone 
$$C_\theta = \{ y \in \mathbb{R}^{n+1} \colon \langle y, e_{n+1}\rangle \geq \cos \theta\}.  $$
If $\tan \theta \leq \sqrt{m-1}$, then $M$ has finite mean exit time.
\end{cor}

\begin{remark}
   A. Atsuji    \cite{atsuji2} proved that $\mathbb{E}[\tau^{p}]<\infty$,  $0<p<1/2$,    for complete minimal submanifolds immersed in non-degenerate cones  of   $\mathbb{R}^{n}_{\raisepunct{.}}$ 
\end{remark}

\begin{remark}{\rm \label{remark 3}Let $\varphi\colon M^m\to  \mathbb{R}^{n}$  be a complete minimal submanifold immersed in a non-degenerate cone $C_{\theta}=[0, \infty)\times_r \Omega_\theta^{n-1} $. If $m=2$, the surface $M$ has finite mean exit time if $\theta\leq \pi/4$. It is straightforward to  show that $\lambda_1(\Omega_{\pi/4}^1)=4$. For $m=3$, $M$ has finite mean exit time if $\theta\leq \arctan \sqrt{2} \doteq \theta_0$, and numerical estimates show that $\lambda_1(\Omega_{\theta_0}^2)\gtrsim 5.85$.}
\end{remark}

In our last result we consider   the tower of exit time moments from  cylindrically bounded submanifolds as in the setting of Theorem \ref{thmMTE_intro}.
From a probabilistic point of view, the $k^{\rm th}$-exit moments of a domain $D\subseteq M$, $k\geq 1$, is the expected value $u_k(x)=\mathbb{E}^{x}(\tau_{\!_D}^{k})$. 
The  $k^\text{th}$-exit moments are deeply intertwined with the geometry of the underlying domain.   For instance, the exit moments are related to the spectrum of the underlying domain, the Faber-Krahn inequality,  isoperimetric inequalities, the heat content, and the technique of symmetrization,  see 
\cite{aizenman-simon,akhizier, burchard-schmuck, mcdonald02, mcdonald03, mcdonald13, pinsky, polya} and references therein. In our next result, we apply Theorem \ref{thmMTE_intro} to derive an upper bound to the $k^\text{th}$-exit moments for domains of submanifolds with small mean curvature.  Precisely, we prove the following result.

\begin{cor}\label{tower_moments_intro}
Let $\varphi \colon M^{m} \rightarrow N^{n-\ell}\times L^\ell$ be a complete  immersed Riemannian submanifold, and assume that  the hypotheses of Theorem \ref{thmMTE_intro} hold. Then the $k^{th}$-exit moment from $D$ is bounded from above by
\begin{equation}u_{_D}^k \leq k! \left(\int_{0}^{r_{_D}} \frac{1}{h^{\eta -1}(\tau)}\int_0^\tau h^{\eta -1} (\xi)d\xi d\tau\right)^k.\label{eq1.4}\end{equation}
In particular, if $\varphi(M)$ is cylindrically bounded, then $M$ has finite $k^\text{th}$-exit moment for all $k\in \{1,2,3, \ldots\}$.
\end{cor}

\begin{remark} When $M^m$ is a minimal submanifold, the $k^{\text{th}}$-exit moment from $D$  is bounded  from above by the $k^{\text{th}}$-exit moment of the ball of radius $r_{\!_D}$ of the $(m-\ell)$-dimensional space form of constant curvature $b$, which is  estimated from above by the right-hand  side of \eqref{eq1.4}, see Section \ref{sectionExit}.
\end{remark}

\noindent \textbf{Acknowledgments:}  The third author thanks Professor Alexander Grigor'yan and the Faculty of Mathematics at the Universit\"at Bielefeld for their hospitality.

\section{Exit time moments}\label{sectionExit}

Let $M$ be a Riemannian manifold and $p\in C^{\infty}((0,\infty)\times M\times M)$ be its heat kernel associated with the operator $\frac{1}{2}\triangle$. Let $X_t$ be the standard Brownian motion on $M$ with transition density  $p$. The mean exit time of  a relatively compact open set $\Omega \subset M$ with smooth boundary $\partial \Omega$, see \eqref{firstexit} satisfies the following Cauchy problem see \cite{dynkin}
\begin{equation}\left\{\begin{array}{rll}
   \triangle E_{\Omega}+1  &=&0 \quad \text{in}\ \ \Omega,  \\[0.2cm]
    E_{\Omega} &=&0 \quad \text{on}\ \ \partial\Omega.
\end{array}\right.
    \end{equation} 
More generally, for any $k\geq 1$ we consider the expected value of the $k^{\text{th}}$ power of the first exit time $\mathbb{E}^{x}[\tau_\Omega^{k}]=u^k_{\Omega}(x)$, which are called exit time moments of $\Omega$. Setting $u^{0}_{\Omega}\equiv 1,$ it follows from the generalized Dynkin formula, see \cite{AK}, that
 the exit moment $u^k_{\Omega}$ satisfy
 \begin{equation}\label{eq_u^k}
 \left\{ \begin{array}{rll}\triangle u^k_{\Omega}+ku^{k-1}_{\Omega}&=&0 \quad {\rm in}\ \ \Omega,\\[0.2cm]
u^k_{\Omega}&=&0\quad {\rm  on}\ \  \partial\Omega,
\end{array}\right.
\end{equation} for $k\geq 0$.
Indeed, an iteration of the above definition shows that $u_\Omega^k$ solves the following Poisson hierarchy
\begin{equation}\left\{ \begin{array}{rll}
\triangle^{k} u^k_{\Omega}+(-1)^{k-1} k!&=&0 \quad {\rm in}\ \ \Omega,\\[0.2cm]
u^k_{\Omega} = \triangle u^k_{\Omega} = \cdots = \triangle^{k-1} u^k_{\Omega} &=&0\quad {\rm  on}\ \  \partial\Omega,
\end{array}\right.\label{eq2.3}\end{equation}
where $\triangle^{k}$ is the composition of the Laplace operator $k$ times (cf. \cite[Sec.2]{kimberly_kinateder_mcdonald}). From   \eqref{eq2.3} it follows that   the $k^{\text{th}}$-exit moment can be expressed as $u_\Omega^k = k!G^k(1)$, where $G^k$ is the iteration of the Green operator, see \cite[Thm. 3.2]{bessa_gimeno_jorge}.

If $\Omega_1 \subset\Omega_2\subset \cdots \subset \Omega_k \subset \cdots$ is an exhaustion  sequence of $M$ by relatively compact subsets, it is clear that $u^{1}_{\Omega_i}=E_{\Omega_i}(x)\leq E_{\Omega_{i+1}}(x)=u^{1}_{\Omega_{i+1}}$ for all $x\in \Omega_i$, and by induction, $u_{\Omega_i}^k(x)\leq u_{\Omega_{i+1}}^k(x)$ for all $x\in \Omega_i$ and $k\geq 0$,  hence we may define the $k^{\text{th}}$-exit moment from $M$ to be   functions $u^k_{M}\colon M \to (0, +\infty]$ given by $$ u^k_{M}(x)=\lim_{i\to \infty}u^{k}_{\Omega_i}(x).$$

In passing, we note that model manifolds support simple and explicit exit time moments. Here, a Riemannian $n$-manifold $\mathbb{M}_{h}^{n} = ([0,+\infty)\times \mathbb{S}^{n-1}, ds^2)$ endowed with the metric $ds^2=dr^{2}+h^2(r)d\theta^2$, with $h(0)=0$, $h'(0)=1$, and $h^{(2k)}(0)=0$, is called a model manifold. Its radial sectional curvature depends only on the polar radius and it satisfies 
$$\Sec_{\mathbb{M}_h^n}(r,\theta) = -\frac{h''(r)}{h(r)}  \doteq -G(r). $$
Thus, $h$ solves the following initial value problem \eqref{cauchy_problem_h_intro}. The $k^{th}$-exit moment of a ball $\Omega_r=B(r) \subset \mathbb{M}_{h}^{n}$ of radius $r>0$, and centered at the origin $o \in \mathbb{M}_{h}^{n}$, is a radial function given  by 
$$ u^k_{\Omega_r}(x)=k\int_{t(x)}^{r}\int_{0}^{\tau}\frac{h^{n-1}(s)u_{\Omega_r}^{k-1}(s)}{h^{n-1}(\tau)} dsd\tau ,$$ as one can check directly. In particular the mean exit time of $\Omega_r$ is $$E_{\Omega_r}(x)=\int_{t(x)}^{r}\int_{0}^{\tau}\frac{h^{n-1}(s)}{h^{n-1}(\tau)}ds d\tau =\int_{t(x)}^{r}\frac{\text{vol}(B(\tau))}{\text{vol}(\partial B(\tau)) }d\tau.$$
Since the $k^{th}$-exit moment $u^k_{\Omega_r}$ is decreasing, we can easily obtain
\begin{eqnarray*}
u^k_{\Omega_r}(x) &\leq & ku_{\Omega_r}^{k-1}(0)\int_{t(x)}^{r}\int_{0}^{\tau}\frac{h^{n-1}(s)}{h^{n-1}(\tau)} dsd\tau \\[0.2cm]
&\leq & k! \left(\int_{0}^{r}\int_{0}^{\tau}\frac{h^{n-1}(s)}{h^{n-1}(\tau)} dsd\tau\right)^k\\[0.2cm]
&=& k! (u^1_{\Omega_r}(0))^{k}.
\end{eqnarray*}
    The above estimate motivates the bound appearing in Theorem \ref{tower_moments_intro}. In particular, if the model manifold is stochastically incomplete then all the $k^{th}$-exit moments are finite. 

For a minimal submanifold $M^m$ immersed in a Riemannian manifold $N^n$ whose sectional curvature satisfies ${\Sec}_N \leq b \leq 0$, it was proved in \cite[Thm.1]{markvorsen} that the mean exit time from a connected domain $\Omega_R \subset M^m$, which is the intersection of a extrinsic regular ball $B_R(o)\subset N^n$ with $M^m$, can be bounded from above by the mean exit time of the regular ball $\mathbb{B}_R^m(0)$ of a model manifold with constant curvature equal to $b$. Precisely, we have
$$E_{\Omega_R} \leq E_{\mathbb{B}_R^m(0)}.$$
Since $u^k_{\Omega_R}$ satisfies \eqref{eq_u^k} for all $k\geq 1$, the comparison principle followed by a simple induction argument then  shows that
$$u^k_{\Omega_R} \leq u^k_{\mathbb{B}_R^m(0)} \quad \text{for all } \ k\geq 1. $$

\section{Proof of Theorems \ref{thmMTE_intro}, \ref{thm_scholium_intro} and Corollary \ref{tower_moments_intro}}\label{sec1}

\subsection{Proof of Theorem \ref{thmMTE_intro}}
Let $\varphi \colon M^m \to N^{n-\ell}\times L^{\ell}$ be an isometric immersion as in the statement of Theorem \ref{thmMTE_intro}. Since the radial sectional curvature of $N^{n-\ell}$ satisfies $\Sec_{N} \leq -G(\rho_N)\leq 0$, the Hessian comparison theorem \cite[Thm.2.3]{prs_vanishing} yields
$$\hess_{_N} \rho_{_N} \geq \frac{h'}{h}(\rho_{_N}) \left(\langle\, ,\, \rangle - d\rho_{_N} \otimes d\rho_{_N}\right) $$
in the sense of quadratic forms in $B_{r}(o)$, where $h \in C^2(\mathbb{R}^+)$ is a positive solution of the Cauchy problem \eqref{cauchy_problem_h_intro}. By assumption \eqref{hyp_curv_media_intro} we can set $\beta = \eta -1>-1$ such that $\sup_{M}\vert H\vert \leq (m-\ell-\beta-1) (h'/h)(r_{\!_D})$. 

Let $E_{_D}\colon \!D \to [0, +\infty)$ be the mean exit time of 
 $D$ and recall that $E_{_D}$ solves the following  Cauchy problem
 \begin{equation}
 \left\{\begin{array}{rcrll}
 \triangle_{_M}E_{_D}&=&-1 &{\rm  in} & D,\\
 E_{_D} & =&0 &{\rm on}& \partial D.
\end{array}\right.
\end{equation}
The function $U\colon [0,r_{_D}]\to [0, +\infty) $ defined by 
$$ U(t)=\int_{t}^{r_{D}}\frac{1}{h^{\beta}(\tau)}\int_{0}^{\tau}h^{\beta}(\xi)d\xi d\tau,$$ 
with $\beta > -1$, is a positive solution for the Cauchy problem
 $$\left\{ \begin{array}{rll} \label{cauchy_prob_U}
 U''(t)+\beta\displaystyle\frac{h'}{h}(t)U'(t)&=&-1,\\[0.2cm]
 U(r_{D})&=&0.
\end{array}\right. $$ 
 
Define $\Phi:[0,+\infty)\to [0, +\infty) $ by
\begin{equation}\label{eq_Phi_h}
    \Phi(t) = \int_{0}^{t}h(\tau)d\tau 
\end{equation}
and observe that $\Phi$ satisfies $\Phi''(t)- [h'(t)/h(t)]\Phi'(t)=0$ for all $t\geq 0$. Set $s=\Phi(t)$ and define $T(s)=U(t)$. Now we have that 
  \begin{eqnarray}
  -1&=&\frac{\partial^2U}{\partial t^2}(t)+\beta\frac{h'}{h}(t)\frac{\partial U}{\partial t}(t)\nonumber \\[0.2cm]
  &=&\frac{\partial^2 T}{\partial s^2}(s)\left(\frac{\partial s}{\partial t}\right)^{\!\!2}\!\!\!(t)+\frac{\partial T}{\partial s}(s)\frac{\partial^2 s}{\partial t^2}+\beta\frac{h'}{h}(t) \frac{\partial T}{\partial s}\frac{\partial s}{\partial t} \nonumber\\[0.2cm]
  &=& \frac{\partial^2 T}{\partial s^2}(s)h^{2}(t)+(1+\beta)h'(t)\frac{\partial T}{\partial s}\cdot \label{eq1.3}
  \end{eqnarray}
 
 Transplant the function $U$  to $B_{N}(r_{_D})\times L$ defining $\widetilde{U}(x,\xi)=U( \rho_N(x))$. We can restrict $\widetilde{U}$ to $M$ considering $\widetilde{U}\circ \varphi\colon M \to \mathbb{R}$.
Letting $\{e_1, \ldots, e_{m}\}$ be an orthonormal basis of $T_yM$, the Laplacian of $\widetilde{U}\circ \varphi$ at $y\in M$, setting $z=\varphi(y)$ is then computed as 
\begin{eqnarray}
\triangle_{M}(\widetilde{U}\circ \varphi) (y)&=&\sum_{i=1}^{m}\hess_{N\times L} \widetilde{U}(z)(e_i, e_i)+ \langle \grad \widetilde{U}, H\rangle(z)\nonumber \\
&=&\sum_{i=1}^{m}\hess_{N} U\circ t(z)(e_i, e_i)+ \langle \grad U\circ t, H\rangle (z)\nonumber\\
 & =&\sum_{i=1}^{m}\hess_{N} T\circ s(e_i, e_i)+ \langle \grad T\circ s(z), H\rangle (z)\label{eqT} \nonumber\\
& =&\sum_{i=1}^{m} (T''(s)\langle \grad s, e_{i}\rangle^{2}(z) + T'(s)\hess_{N}\,s(z)(e_{i},e_{i})) \\ 
&&\,\,\,+\,\,\,\langle \grad s,\, H\rangle(z). \nonumber
\end{eqnarray}
Let $\{\partial/\partial t, \partial/\partial \theta_{1}, \ldots,\partial/\partial \theta_{n-\ell-1}\}$ and $\{\partial/\partial \xi_1,\ldots, \partial/\partial \xi_{\ell}\}$ be orthonormal bases for $T_{\pi_{1}(\varphi(y))}N$  and  $T_{\pi_{2}(\varphi(y))}L^\ell$, respectively. Writing   
  \begin{eqnarray}\label{basis} e_{i}&=&\alpha_{i}\cdot \partial/\partial t+\sum_{j=1}^{\ell}\beta^{i}_{j}\cdot\partial/\partial \xi_{j} + \sum_{j=1}^{n-\ell-1}\gamma_{j}^{i} \cdot\partial/\partial \theta_{j},
  \end{eqnarray}
 where
 \begin{eqnarray} \alpha_{i}^{2}+\sum_{j=1}^{\ell}(\beta^{i}_{j})^{2}+\sum_{j=1}^{n-\ell-1}(\gamma_{j}^{i})^{2} = 1,
 \end{eqnarray}
we can estimate $ \sum_{i=1}^{m}\hess_{N} \,s (e_{i},e_{i})$ using the Hessian comparison theorem 

\begin{eqnarray}\label{hessiano-de-s-A}
\sum_{i=1}^{m}\hess_{N} \,s (e_{i},e_{i}) &=& \hspace{-1mm} \Phi''(\rho_N)\sum_{i=1}^{m}\langle \grad \rho_N, e_{i}\rangle^{2} +  \Phi'(\rho_N)\sum_{i=1}^{m}\hess_{N}\, \rho_N (e_{i},e_{i})\nonumber  \\
&\geq & \Phi''(\rho_N)\sum_{i=1}^{m}\alpha_{i}^{2} + \Phi'(\rho_N)\sum_{i=1}^{m}\sum_{j=1}^{n-\ell-1}(\gamma_{j}^{i})^{2} \frac{h'}{h}(\rho_N) \nonumber\\
 &= & \Phi''(\rho_N)\sum_{i=1}^{m}\alpha_{i}^{2} +  \Phi'(\rho_N)\frac{h'}{h}(\rho_N)\sum_{i=1}^{m}\left(1-\alpha_{i}^{2}-\sum_{j=1}^{\ell}(\beta^{i}_{j})^{2}\right)\nonumber \\
&=& \Phi'(\rho_N)\frac{h'}{h}(\rho_N)\left(m-\sum_{i=1}^{m}\sum_{j=1}^{\ell}(\beta_{j}^{j})^{2}\right)\nonumber \\
&\geq &(m-\ell)h'(\rho_N).
\end{eqnarray}

We now need the following lemma from \cite{markvorsen}. We will present a simple proof for the reader's convenience.

\begin{lem}\label{lemmaM} Since $G\geq 0$, the function $T$ satisfies
$$ T'(s)<0 \quad \text{and} \quad T''(s)\geq 0 \quad \text{for all } \ s \geq 0. $$
\end{lem}
\begin{proof}
From $\displaystyle \frac{\partial U}{\partial t}=\frac{\partial T}{\partial s}\frac{\partial s}{\partial t}$ together with
$\displaystyle\frac{\partial U}{\partial t}(t)=-\frac{1}{h^{\beta}}(t)\int_{0}^{t}h^{\beta}(\xi)d\xi<0$ and $\displaystyle \frac{\partial s}{\partial t}=h(t)$ we have  that 
$$\displaystyle\frac{\partial T}{\partial s} = - [h(t)]^{-(\beta + 1)} \int_{0}^{t}h^\beta (\tau)d\tau <0 \quad \text{for } \ t>0.$$
For the second derivative $T''(s)$ we rewrite equation \eqref{eq1.3} as
  \[h^2(t)\frac{\partial^2 T}{\partial s^2}  = -1 + (1+\beta)h'(t)[h(t)]^{-(\beta + 1)} \int_{0}^{t}h^\beta (\tau)d\tau .\]
  Since $G\geq 0$, we have $h'' \geq Gh\geq 0$ then $h'$ is non-decreasing. Therefore 
  \begin{eqnarray*}
  h^2(t)\frac{\partial^2 T}{\partial s^2} &\geq & -1 + (1+\beta)[h(t)]^{-(\beta + 1)} \int_{0}^{t}\frac{[h^{\beta+1}]'(\tau)}{\beta + 1} d\tau \\[0.2cm]
  &=& -1 + 1 = 0.
  \end{eqnarray*}
This completes the proof of the lemma.
\end{proof}

From   \eqref{eqT}, \eqref{hessiano-de-s-A} and Lemma \ref{lemmaM} we have

 \begin{eqnarray}
 \triangle_{D}\widetilde{U}\circ \varphi(x)
  &=&  \sum_{i=1}^{m}\left[ T''(s)\langle \grad s, e_{i}\rangle^{2} + T'(s)\hess_{N}\,s(e_{i},e_{i})\right]\nonumber\\[0.2cm]
  && + T'(s) \langle \grad s,\, H\rangle\nonumber \\[0.2cm]
  &\leq & T''(s)\vert \grad_{N} s\vert^{2} +  T'(s)\cdot (m-\ell) h'(\rho_N)\nonumber \\[0.2cm]
  & & +T'(s)h(\rho_N)\langle \grad \rho_N, \,H\rangle\nonumber \\[0.2cm]
 &\leq &  T''(s)\vert \grad_{N} s\vert^{2} +  T'(s)\cdot (m-\ell) h'(\rho_N)\nonumber \\[0.2cm]
  & & -T'(s)\cdot (m-\ell- \beta-1) h'(\rho_N)\nonumber \\[0.2cm]
  &= &  T''(s)\vert\partial s/\partial \rho_N\vert^{2} +  T'(s)\cdot (\beta+1) h'(\rho_N)\nonumber \\[0.2cm]
  &=& T''(s)h^{2}(\rho_N) +  T'(s)\cdot (\beta+1) h'(\rho_N) \nonumber\\[0.2cm]
  &=& -1 . \nonumber
 \end{eqnarray}
Since $(\widetilde{U}\circ \varphi -E_{\!_D})\vert_{\partial D}= \widetilde{U}\circ \varphi\vert_{\partial D}\geq 0$ and $\triangle_{M}E_{\!_D} = -1$ we conclude by the comparison principle that
$$E_{\!_D}(x) \leq \widetilde{U}\circ \varphi(x) \quad \text{for every } \ x \in D.$$

If $M$ is cylindrically bounded we may take an exhaustion by compact sets $\Omega_k$ to obtain 
$$ E_{\Omega_k}(x) \leq \widetilde{U}\circ \varphi(x),$$
where $\widetilde U$ is defined using a fixed radius $r$. The result then follows letting $k \rightarrow + \infty$.

\subsection{Proof of Corollary \ref{tower_moments_intro}}

As in the proof of Theorem \ref{thmMTE_intro}, for a given $\beta > -1$, let $U\colon [0,r_{\!_D}]\to [0, \infty) $ be the function defined by 
$$ U(t)=\int_{t}^{r_{D}}\frac{1}{h^{\beta}(\tau)}\int_{0}^{\tau}h^{\beta}(\xi)d\xi d\tau,$$ 
which is a positive solution for the Cauchy problem \eqref{cauchy_prob_U}. For each $k\ge 1$ we set $C_k \doteq \max_D u_D^{k-1}$, and define $U_k(t) = k C_k U(t)$. By definition of $u_D^k$ we know that $U_1 \equiv U$ and from the mean exit time estimate in Theorem \ref{thmMTE_intro}, $U_2$ satisfies
$$U_2 \leq 2 \max_D E_D U \leq 2 U(0)^2.$$
We then assume by induction that $U_{k-1} \leq (k-1)! U(0)^{k-1}$.
We note that, for every $j\geq 1$, $U_{j} > 0$ on $\partial D$ and $\triangle_M U_{j} = jC_j\triangle_M U \leq -j u_D^{j-1}$ in $D$. So, we can apply the comparison principle to show that $u_D^{j} \leq U_{j}$. Therefore,
$$u_D^{k} \leq k C_k \leq k U_{k-1}(0) \leq k! U(0)^k, $$
as claimed.

\subsection{Proof of Theorem \ref{thm_scholium_intro}}

Let $N$ be a Cartan-Hadamard Riemannian manifold whose sectional curvature satisfies $\Sec_N \leq -b < 0$. Let $\sigma \colon [0,+\infty) \rightarrow N$ be a unit speed ray and define 
\begin{equation}\label{def_b_t}
    b_t(x) = \rho_{\sigma(t)} - t,
\end{equation}
where $\rho_{\sigma(t)}(x) = \dist_N(x,\sigma(t))$. The Busemann function associated to the ray $\sigma$ is given by 
$$b_\sigma(x) = \lim_{t\rightarrow + \infty} b_t(x).$$
We recall that the horoball of radius $r>0$ in $N$ determined by $\sigma$ is the set
$$\mathfrak{B}_{\sigma,r}  = \{x \in N \colon b_\sigma(x) \leq r\},$$
and for any $\ell$-dimensional manifold $L$ the region $\mathfrak{B}_{\sigma,r}\times L \subset N\times L$ is a (generalized solid) horocylinder in $N\times L$. 

In \cite[Thm.2.3]{bessa_lira_pigola_setti}, the authors established a Hessian comparison theorem for the Busemann function under pinched sectional curvature  $-a \leq \Sec_N \leq -b< 0$, for $a, b \in \mathbb{R}^{+}$. Here below, we just observe that to obtain a lower bound on the Hessian of the Busemann function, the pinched curvature assumption is not essential.

\begin{lem}\label{busemman_hessian}
    If $N$ is a Cartan-Hadamard manifold with sectional curvature bounded above by $\Sec_{N}\leq -b < 0$, then
  \begin{equation}\label{hessian_busemann}
  \hess\, b_\sigma \geq \sqrt{b}\left(\langle\, ,\, \rangle - db_\sigma \otimes db_\sigma\right) 
  \end{equation}
in the sense of quadratic forms. 
\end{lem}
\begin{proof}
 We follow  the proof of \cite[Thm.2.3]{bessa_lira_pigola_setti}. Fix a ball $B_R \subset N$ and apply the Hessian comparison theorem \cite[Thm.2.3]{prs_vanishing} for the function $b_t$ to obtain 
 $$\hess b_t \geq \sqrt{b}\coth(\sqrt{b}\,\rho_{\sigma(t)}) \left\{ \langle\, ,\, \rangle - db_t \otimes db_t\right\} \,\,\, \text{in } \ B_R,$$
 for every $t\geq T>0$. Since the distance function  $b_{\sigma(t)}$ is smooth in $B_R$, the above point-wise inequality also holds in the sense of distributions, i.e.,
 $$\int b_{\sigma(t)}\{ \diver(V \diver V) + \diver \nabla_{V}V\} \geq \sqrt{b}\int \coth(\sqrt{b}\rho_{\sigma(t)})\left\lbrace \vert V \vert^2 - \langle \nabla b_{\sigma(t)},V\rangle^2\right\rbrace$$
 for every vector field $V$ compactly supported in $B_R$. To conclude, one would like to pass the limit in the above inequality to obtain
 $$\int b_{\sigma}\{ \diver(V \diver V) + \diver \nabla_{V}V\} \geq \sqrt{b}\int \left\lbrace \vert V \vert^2 - \langle \nabla b_{\sigma},V\rangle^2\right\rbrace ,$$
 which turns out to be equivalent to the point-wise inequality \eqref{hessian_busemann}. Indeed, the above convergence is valid whence one has proved that $\nabla b_{\sigma(t)}$ converges locally uniformly to $\nabla b_\sigma$. At this point we will proceed differently from \cite[Lemm.2.1]{bessa_lira_pigola_setti} observing that the proof of \cite[Prop.3.1]{heintze-hof}, which is inspired in an unpublished paper by P. Eberlein, shows that 
$$\nabla b_{\sigma(t)} \rightarrow \nabla b_\sigma \in C^1,$$
locally uniformly under the solely upper bound of the sectional curvature. For the sake of completeness, we are going to present the argument with details.  

Let $\sigma$ be a unit speed ray in $N$ and denote by $p_n = \sigma(n)$, for $n \in \mathbb{N}$. Let $X_n = - \nabla b_{n}$ be the radial field in the direction of $p_n$, where $b_n$ is the function defined in \eqref{def_b_t}. Recall that $X_n$ is smooth on $N\backslash \{p_n\}$. To prove that $X_n$ converges locally uniformly to $X = \nabla b_\sigma$, let $K\subset N$ be a compact set such that $p_n \notin K$ for every $n\geq n_0$, for some $n_0 \in \mathbb{N}$. Given $p \in K$ and $n\geq n_0$, consider a sequence of unit speed segments $\sigma_n \colon [0,L_n] \rightarrow N$ connecting $p$ to $p_n = \sigma(n)$ which converges to an asymptote $\tilde{\sigma}$ for $\sigma$ from $p$. Thus, 
\begin{equation}\label{angle_unif-conv}
\vert X_n - X\vert(p) = \vert \sigma'_n(0) - \tilde{\sigma}_{p}\vert \rightarrow 0
\end{equation}
uniformly on $K$, if the angles $\angle_p(p_n,\tilde{\sigma})$ do so. Indeed,
$$\vert \sigma'_n(0) - \tilde{\sigma}'(0)\vert^2 = 2(1 - \langle \sigma'_n(0), \tilde{\sigma}'(0)\rangle).$$ 
Now, since $N$ is Cartan-Hadamard it has no focal points, and by \cite[Prop.3]{Osullivan} there holds
$$\dist_N(p_n,\tilde{\sigma}) \leq \max\{ \dist(\sigma(0),p) \colon p \in K\}.$$
This uniform boundedness of the distances yields the locally uniformly convergence of the angles $\angle_p(p_n,\tilde{\sigma})$ to zero. Therefore, $X_n$ converges locally uniformly to $X$, and the proof is complete.
\end{proof}

With the above preparation, the proof of Theorem \ref{thm_scholium_intro} follows the same steps as in the proof of Theorem \ref{thmMTE_intro}. The main point is to consider the Busemann function $b_\sigma$ defining the horoball in place of the distance function $\rho_N$.

We start noticing that, in view of Lemma \ref{busemman_hessian}, as in  \eqref{hessiano-de-s-A} the Hessian of $s = \Phi(b_\sigma)$, with $\Phi$ defined in \eqref{eq_Phi_h}, becomes
\begin{eqnarray*}
\hess_{N}\, s (e_{i},e_{i}) &=& \Phi''(b_\sigma)\alpha_{i}^{2} + \Phi'(b_\sigma)\sum_{j=1}^{n-\ell-1}(\gamma_{j}^{i})^{2}\hess_{N} b_\sigma(\partial/\partial \theta_{j},\partial/\partial \theta_{j}) \nonumber\\
&\geq & \Phi''(b_\sigma)\alpha_{i}^{2} +  \Phi'(b_\sigma)\frac{h'}{h}(b_\sigma)\left(1-\alpha_{i}^{2}-\sum_{j=1}^{\ell}(\beta^{i}_{j})^{2}\right)\nonumber \\
&=& h'(b_\sigma)\left(1-\sum_{j=1}^{\ell}(\beta_{j}^{j})^{2}\right).
\end{eqnarray*}
Thus,
\begin{eqnarray}\label{ineq_hessian_busemann}
\sum_{i=1}^{m} \hess_{N}\, s (e_{i},e_{i}) \geq (m-\ell)h'(b_\sigma).
\end{eqnarray}

As in the end of the proof of Theorem \ref{thmMTE_intro}, we substitute \eqref{ineq_hessian_busemann} in \eqref{eqT} and apply Lemma \ref{lemmaM} together the fact that $\vert \nabla b_\sigma\vert \leq 1$ to obtain
$$ \triangle_{D}\widetilde{U}\circ \varphi(x) \leq -1.$$
The proof then follows by the comparison principle.

\section{Proof of Theorem \ref{thm_wedge_intro}}

The proof follows the same comparison strategy used in the previous proofs, that is, we need to show the existence of a smooth function $u \colon M^m \to \mathbb{R} $ satisfying $\triangle_ M u\leq -1$. Consider $u \colon N\times L\times P \rightarrow \mathbb{R}$ given by 
$$ u(x,y,z) = \frac{\alpha^2 \rho_L^2(y)}{2\gamma} - \frac{\rho_N(x)^2}{2\gamma},$$
with $\gamma = m - (\alpha^2+1)\ell - k$. At each point $p \in M$, let $\{e_i\}_{i=1}^{m}$ be an orthonormal basis for $T_p M$ and write
\begin{equation*}e_i = a_i \nabla \rho + \sum_{j=1}^{n-1} b_i^j \frac{\partial}{\partial\theta_j} + c_i \nabla t + \sum_{j=1}^{\ell-1} d_i^j \frac{\partial}{\partial\sigma_j} + \sum_{j=1}^{k} e_i^j \frac{\partial}{\partial\xi_j},
\end{equation*}
where 
$$\{\nabla \rho, \partial/\partial\theta_1,\ldots,\partial/\partial\theta_{n-1}\}, \{\nabla t, \partial/\partial\sigma_1,\ldots,\partial/\partial\sigma_{\ell-1}\} \ \ \text{and} \ \ \{\partial/\partial\xi_1,\ldots,\partial/\partial\xi_{k}\}$$ are orthonormal bases to $T_{\pi_1(\varphi)}N$, $T_{\pi_2(\varphi)}L$ and $T_{\pi_3(\varphi)}P$, respectively. We observe that
$$ a_i^2 + \sum_{j=1}^{n-1} (b_i^j)^2 + c_i^2 + \sum_{j=1}^{\ell-1} (d_i^j)^2 + \sum_{j=1}^{k}(e_i^j)^2 = 1.$$

We now compute the Laplacian of $u$ transplanted to $M$ when $\ell \geq 2$, the other case $\ell = 1$ can be treated similarly. We have

\begin{eqnarray*}
\triangle_M u &=& \sum_{i=1}^{m} \hess u (e_i,e_i) \\[0.2cm]
&\leq & \frac{\alpha^2}{\gamma}\sum_{i=1}^{m} \left(c_i^2 + \sum_{j=1}^{\ell-1}(d_i^j)^2\right) - \frac{1}{\gamma}\sum_{i=1}^{m} \left(a_i^2 + \sum_{j=1}^{n-1}(b_i^j)^2\right)\\[0.2cm]
&=& \frac{\alpha^2}{\gamma}\sum_{i=1}^{m} \left(c_i^2 + \sum_{j=1}^{\ell-1}(d_i^j)^2\right) - \frac{1}{\gamma}\sum_{i=1}^{m} \left(1 - c_i^2 - \sum_{j=1}^{\ell-1}(d_i^j)^2 - \sum_{j=1}^{k}(e_i^j)^2\right)\\[0.2cm]
&=& \frac{1}{\gamma}\left[(\alpha^2 +1)\sum_{i=1}^{m} \left(c_i^2 + \sum_{j=1}^{\ell-1}(d_i^j)^2\right) + \sum_{i=1}^{m}\sum_{j=1}^{k}(e_i^j)^2 - m\right]\\[0.2cm]
&\leq & \frac{1}{\gamma}\left[(\alpha^2 +1)\ell +k - m\right]\\[0.2cm]
&\leq & -1.
\end{eqnarray*} This finishes the proof.

\section{Appendix I}

As reported in Remark \ref{remark 3} we are going to present here some lower estimates for the first Dirichlet eigenvalue $\lambda_1(B_r)$ for  the ball $B_r \subset \mathbb{S}^{m-1}(1)$  centred at the north pole and with polar radius $0<r<\pi/2$. The following Barta's criterion to estimate the first eigenvalue was obtained in \cite{barroso_bessa}   
\[
 \lambda_1(B_r)\geq \inf_{t\in [0,r]}\left[\displaystyle\frac{u(t)}{\displaystyle \int_{t}^{r}\int_{0}^{\tau}\displaystyle\frac{\sin(s)^{m-2} u(s)}{\sin(\tau)^{m-2}}dsd\tau}\right]\cdot\]

Using Maple we obtain numerical estimates for the infimum appearing in the above inequality with the choice $u(t)=\cos(t\pi/2r)$. In what follows we are going to compare the sufficient condition $r \leq \arctan(\sqrt{m-1})$ from Corollary \ref{cor_cone_intro} with the hypothetically sufficient condition $\lambda_1(B_r) > 2m$ from Theorem \ref{prop_eigenvalue_intro}.
\begin{enumerate}
    \item[1)] $m=3$ and $r = \arctan(\sqrt{2}) = \frac{\pi}{3}\colon $  
    $$\lambda_1(B_r) \gtrsim 5.85, \quad \text{which is not bigger than $2m = 6$}.$$
    \item[2)] $m=4$ and $r = \arctan(\sqrt{3})\colon $
    $$\lambda_1(B_r) \gtrsim 7.60, \quad \text{which is not bigger than $2m = 8$}.$$
    \item[3)] $m=5$ and $r = \arctan(2)\colon $
    $$\lambda_1(B_r) \gtrsim 9.28, \quad \text{which is not bigger than $2m = 10$}.$$
\end{enumerate}

\section{Appendix II}

We are going to address the problem of characterizing finiteness of the mean exit time in terms of the first Dirichlet eigenvalue in general warped product spaces over $[0,+\infty)$. 

Let $P^\ell = [0,+\infty)\times_w L^{\ell -1}$ be a Riemannian warped product where $L^{\ell -1}$ is a compact manifold without boundary and  $w \colon [0,+\infty) \rightarrow [0,+\infty)$ is the warping function. Let $\Omega \subset L$ be an open domain and consider $\phi$ a positive first eigenfunction of $\Omega$ for the Laplacian operator on $L$, that is, $\triangle_L \,\phi + \lambda \phi = 0$ in $\Omega$. We are going to consider a cone over $\Omega$ defined as follows
$$C_\Omega = \{(r,\theta) \in P^\ell \colon \theta \in \Omega, r>0\}.$$

 Given a non-negative function $u \colon [0,+\infty) \rightarrow [0,+\infty)$, an easy computation shows that (see e.g. \cite{gomes_marrocos_2019})
$$\triangle_{P}(u\phi) = \phi\left(u'' + (\ell-1) u' \frac{w'}{w} - \frac{\lambda}{w^2}u \right). $$
Arguing as in the proof of \cite[Prop.43]{pessoa-pigola-setti}, up to shrinking $\Omega$ and normalizing $\phi$, we can assume that $\inf \phi > 0$ such that $\inf \phi = 1$. In order to use the comparison principle to estimate the mean exit time of $C_\Omega$ from above, one needs to show the existence of a non-negative function $u$ satisfying
\begin{equation}\label{eq_u_solve_ode}
    u'' + (\ell-1) u' \frac{w'}{w} - \frac{\lambda}{w^2}u \leq - 1.
\end{equation}
From \cite[Thm.33]{pessoa-pigola-setti}, to show that $C_\Omega$ has finite mean exit time it is sufficient to prove the finiteness of the mean exit time of the truncated cone $C_{\Omega}(r_0) = [r_0,+\infty)\times_{w}\Omega$, with zero Dirichlet boundary data at $r_{0}$, for some $r_{0}>0$ fixed. That is, we only need to build a positive supersolution for the equation $\Delta u = -1$ on $C_{\Omega}(r_0)$.

One such positive solution $u \colon [r_0,+\infty) \rightarrow [0,+\infty)$ for problem \eqref{eq_u_solve_ode} is obtained as follows. Set 
$$ \psi(t) = \int_{{0}}^t w^{\ell-1}(s)ds$$
and define the non-negative function
$$u(t) = (c\lambda - 1) \int_{{0}}^t \frac{\psi}{\psi'}(s)ds$$
for some constant $c > \lambda^{-1}$ such that 
$$\inf_{t\geq r_{0}} \frac{u}{w^2}(t) \geq c.$$
This last condition gives some restrictions on the warping function $w$. Indeed, this latter inequality can be rewritten as
\begin{equation} \label{cond_w_warping} \frac{1}{w^2(t)}\int_{{0}}^t \frac{\int_{{0}}^s w^{\ell-1}(\tau)d\tau}{w^{\ell-1}(s)}ds \geq \frac{c}{c\lambda - 1} \qquad \forall\, t\geq r_{0}.
\end{equation}

If $w(t) \rightarrow + \infty$ as $t \rightarrow + \infty$, using L'Hospital rule we can derive the following asymptotic condition 
$$\lim_{t\rightarrow + \infty} (\ell w'(t)^2 + w(t)w''(t)) \leq \frac{c\lambda - 1}{2c}.$$
We can rewrite this inequality using the radial sectional-curvature-like function $K(t) = -w''(t)/w(t)$ and the radial mean-curvature-like function on the slices $H(t) = w'(t)/w(t)$ as
$$\lim_{t\rightarrow + \infty} w^2(t)(\ell H(t)^2 - K(t)) \leq \frac{c\lambda - 1}{2c}.$$
So, a necessary condition is given by 
$$\ell H(t)^2 - K(t) = O(w(t)^{-2}). $$
We recall that, if $L = \mathbb{S}^{\ell -1}$ is the round sphere, then the warped product space $P^\ell$ must be stochastically complete in order to use L'Hospital rule in \eqref{cond_w_warping} at infinity.

The above discussion gives rise to the following generalization of \cite[Prop.43,i)]{pessoa-pigola-setti}, which can be viewed as a dual version of \cite[Thm.44]{pessoa-pigola-setti}.

\begin{theorem}
Let $P^\ell$ be an $\ell$-dimensional smooth Riemannian manifold. Assume that there exists a compact set $K \subset P$ such that $P\backslash K$ is isometric to a truncated warped product cone $C_\Omega(r) = [r,+\infty)\times_w \Omega$, for some domain $\Omega$ of a compact manifold $L^{\ell-1}$, and $r>0$. If the first Dirichlet eigenvalue $\lambda=\lambda(\Omega)$ satisfies \eqref{cond_w_warping} for some $c>\lambda^{-1}$, then $P$ has finite mean exit times.
\end{theorem}

\begin{ex}{\rm
Let $w(t) = \alpha + kt$ for some $\alpha \geq 0$ and $k\geq 0$ with $(\alpha, k) \neq (0,0)$. In these simple cases, \eqref{cond_w_warping} reads as 
$$ \frac{c}{c\lambda - 1} \leq \frac{1}{2\ell k^2}, $$
independent of $\alpha$, i.e.
$$ c \geq \frac{1}{\lambda - 2\ell k^2}.$$
The condition for finite mean exit times from the cone $C_\Omega(r)$ is 
thence given by
$$\lambda(\Omega) > 2\ell k^2,$$
which coincides with that obtained in \cite[Prop.43]{pessoa-pigola-setti} when $\alpha = 0$ and $k=1$. Moreover, from these general examples we observe: Firstly, if $k$ is large, so that the cone is  wide, then the first eigenvalue of $\Omega$ must also be large in order to produce finite mean exit times from the cone; Secondly, if $k$ is small, so that the cone is narrow, then only a small eigenvalue of $\Omega$ is needed; Thirdly, in the limit, where $k=0$ (and $\alpha >0$), and since $\lambda(\Omega)$ is always positive,
we recover the fact that every cylindrical ``cone'' $C_\Omega(r)$ has finite mean exit times.}
\end{ex}

\end{document}